\documentclass[a4paper,twoside,openbib,11pt]{article}

\usepackage{amsfonts}
\usepackage{amsmath}
\usepackage{amsthm}
\usepackage{amssymb}
\usepackage[all]{xy}
\usepackage[latin1]{inputenc}
\usepackage{graphicx}
\usepackage{color}
\usepackage{yfonts}

\input xy
\xyoption{all}

\newtheorem{theorem}{Theorem}[section]
\newtheorem{corollary}[theorem]{Corollary}

\newtheorem{proposition}[theorem]{Proposition}
\theoremstyle{definition}
\newtheorem{definition}[theorem]{Definition}

\newtheorem{example}[theorem]{Example}

\newcommand{\Ker}{\operatorname{ker}}

\newcommand\coker{\operatorname{coker}}
\newcommand\Img{\operatorname{\rm Im\,}}

\newcommand\iso{\kern.35em{\raise3pt\hbox{$\sim$}\kern-1.1em\to}\kern.3em}

\title{A note on the kernel of a pair of linear maps
\thanks{Partially supported by INCIBE. Ministerio de Industria, Spain.
 }
\author{Miguel V. Carriegos\footnote{RIASC, Universidad de Le\'on, SPAIN, mail to: {\texttt miguel.carriegos@unileon.es}}\qquad Noem\'{i} DeCastro\footnote{Departamento de Matem\'aticas, Universidad de Le\'on, SPAIN, mail to: {\texttt ncasg@unileon.es}} \qquad Ángel Luis Muñoz Castañeda\footnote{Institut f\"ur Mathematik, Freie Universit\"at, Berlin, GERMANY, mail to: {\texttt angel@math.fu-berlin.de}}}
}
\begin{document}

\maketitle
\begin{abstract}
The kernel of a pair of linear maps is studied in the framework of commutative ring theory with application to behavioral perspective of dynamical systems.\end{abstract}

{\textsl Keywords:} pair of linear maps, pencil of matrices, convolution code, linear system, hereditary ring.

{\textsl 2010 MSC: } 93B10, 15A21, 13C10.

\section{Introduction}
This paper is devoted to the study of kernels of pairs of linear maps. This notion extend some tools used in systems theory, convolutional codes and Boolean networks \cite{Rosenthal}, \cite{Puzzles}.

In fact if $\mathcal{A}=R^{\mathbb{Z}}$ is the $R$-algebra ($R$ a ring) of sequences of elements of $R$; $A\in R^{n\times n}$, $B\in R^{n\times m}$ are matrices and $\sigma$ is the shift operator $\sigma(x(t))=x(t+1)$. Then set 
$$\Ker(\sigma\mathbf{I}-A,B)=\{(x(t),u(t))\in\mathcal{A}^n\times\mathcal{A}^m\mid \sigma(x(t))=x(t+1)=A x(t)+B u(t)\}$$
collects the trajectories of linear system. On the other hand, it has also been defined
$$\Ker(\sigma\mathbf{I}-A\mid B)=\{u(t)\in\mathcal{A}^m\mid \exists x(t)\in \mathcal{A}^n: \sigma(x(t))=x(t+1)=A x(t)+B u(t)\}
$$
which is central in the definition of convolutional codes because it collects the codewords of convolutional code defined by linear system $(A,B)$. In order to attack this problem we are studying kernels of pairs of matrices with entries in polynomial rings $R[z]$.

Now all along this paper $R$ will denote a commutative ring with identity. Usually $R$ will be a $\mathbb{F}$-algebra or even a field. We develop our results in a general framework and claim additional properties or structure when necessary.

\begin{definition}
Let $N$ and $M_i$, $i=1,2$ be $R$-modules and $f_i:M_i\rightarrow N$ be $R$-linear maps. We define
$$
\Ker(f_1\mid f_2)=\{m_2\in M_2\mid \exists m_1\in M_1 : f_1(m_1)+f_2(m_2)=0\}
$$
while we set
$$
\Ker(f_1,f_2)=\Ker(f_1\oplus f_2:M_1\oplus M_2\rightarrow N)=\{(m_1,m_2)\in M_1\oplus M_2:f(m_1)+f_2(m_2)=0\}
$$
as usual.
\end{definition}

The paper is organized as follows: Section $2$ deals with main properties of $\Ker(f_1\mid f_2)$ which follow from the fact that $\Ker(f_1\mid f_2)$ is the kernel of certain linear map and it is also the cockernel of another linear map. We also point out several properties of $\Ker(f_1\mid f_2)$ when maps $f_i$ have nice properties.

Section $3$ is devoted to give a explicit factorization of $\Ker(f_1\mid f_2)$ in terms of usual kernels $\Ker(f_1,f_2)$ and $\Ker(f_1)$ when involved $R$-modules are hereditary or in particular when commutative ring $R$ is hereditary (that is when submodules of projective $R$-modules are again projective). Note that the class of hereditary rings contains several interesting classes of commutative rings like for instance fields, principal ideal domains, Dedekind domains, von Neumann regular rings and Boolean rings.

Next we study scalar extensions $R\rightarrow S$ in section $4$. If $R$-algebra $S$ happens to be $R$-flat then scalar extension of $\Ker(f_1\mid f_2)$ is given by the kernel of pair of extended maps $\Ker(f_1\otimes\mathbf{I}_S\mid f_2\otimes\mathbf{I}_S)$. This fact is used extensively in section $5$ to develop the case of product rings $R\cong R_1\times\cdots\times R_s$. We conclude by giving some results related to the kernel of pairs of polynomial matrices which are applicable both in behavioral linear systems and convolutional codes.


\section{The kernel of a pair. Definition and properties}
Let $R$ be a commutative ring;  $M_1,M_2,N$ are $R$-modules and let $f_ j :M_j\rightarrow N$ be $R$-linear maps. 

\begin{definition}\label{DefKerKer}
We denote by $\Ker(f_1\mid f_2)$ the subset of all $m_2\in M_2$ such that there exists $m_1\in M_1$ with the property $f_1(m_1)+f_2(m_2)=0$; thats to say
$$
\Ker(f_1\mid f_2)=\left\{m_2\in M_2 \mid \exists m_1\in M_1: f_1(m_1)+f_2(m_2)=0\right\}.
$$
\end{definition}
This is a generalization of so-called kernel of a pair of morphisms as given in \cite{Rosenthal}, \cite{Puzzles}. Note that it is quite clear that $\Ker(f_1\mid f_2)$ is a $R$-submodule of $M_2$. It is also straightforward that

\begin{proposition}\label{f2-1Imf1}
$\Ker(f_1\mid f_2)=f_2^{-1}(\Img(f_1))$.
\end{proposition}
\begin{proof}
If $m_2\in\Ker(f_1\mid f_2)$ then $f_1(m_1)+f_2(m_2)=0$ for some $m_1$; that is, $f_2(m_2)\in\Img(f_1)$ or $m_2\in f_2^{-1}(\Img(f_1))$
\end{proof}

We are also interested in presenting $\Ker(f_1\mid f_2)$  as the kernel of a $R$-linear map and as the cokernel of another $R$-linear map. Let's denote by $p_i:N\rightarrow N/\Img(f_i)$ the natural quotient map of identity sending $n\mapsto n+\Img(f_i)$. Then it is straightforward that

\begin{proposition}
$\Ker(f_1\mid f_2)=\Ker(p_1\circ f_2:M_2\rightarrow N/\Img(f_1))$
\end{proposition}

\begin{proof}
$x\in\Ker(p_1\circ f_2:M_2\rightarrow N/\Img(f_1))\Leftrightarrow f_2(x)\in\Img(f_1)$
\end{proof}

On the other hand put the linear map $(f_1\, ,\, f_2):M_1\oplus M_2\rightarrow N$ sending $(m_1,m_2)\mapsto (f_1\, ,\, f_2)(m_1,m_2)=f_1(m_1)+f_2(m_2)$. Note that $\Ker(f_1)\subseteq\Ker(f_1\, ,\, f_2)$ because if $m\in\Ker(f_1)$ then $f_1(m)=0$; hence $(f_1\, ,\, f_2)(m,0)=0$ and consequently $(m,0)\in\Ker(f_1\, ,\, f_2)$. Moreover one has

\begin{theorem}\label{TeoremaConucleo}
$\Ker(f_1\mid f_2)=\coker[\Ker(f_1)\hookrightarrow\Ker(f_1\, ,\, f_2)]$
\end{theorem}

\begin{proof}
The projection onto the second factor $\pi_2:M_{1}\oplus M_{2}\rightarrow M_{2}$ restricts to the onto map $\Ker(f_1,f_2)\stackrel{\pi_2}{\longrightarrow}\Ker(f_1\mid f_2)$ whose kernel consist of all pairs satisfying $\pi_2(x_1,x_2)=0$ and therefore $x_2=0$ by definition of $\pi_2$. So, a fortiori, $f_1(x_1)=0$, and consequently we are done because  $$\Ker\left[\Ker(f_1\,,\,f_2)\stackrel{\pi_2}{\longrightarrow}\Ker(f_1\mid f_2)\right]=\Ker(f_1)$$ 
That is, the following sequence is short exact 
\begin{align*}
0\rightarrow \Ker(f_1)&\hookrightarrow \Ker(f_1, f_2)\rightarrow  \Ker(f_1 \mid f_2)\rightarrow 0\\
m_{1}&\mapsto \ \ (m_{1},0)\\
& \ \ \ \ \ \ \ (m_{1},m_{2}) \mapsto \ \ m_{2}
\end{align*}
\end{proof}

The following properties of the kernel of a pair of linear maps, may be easily derived.
\begin{proposition} Consider the $R$-linear maps $f_i:M_i\rightarrow N$. Denote by $\mathbf{0}$ and $\mathbf{I}$ respectiveli the zero linear map and the identity map. Then
\begin{itemize}
\item[(i)] $\Ker(\mathbf{0}\mid f_2)=\Ker(f_{2})$
\item[(ii)] If $f_1$ is onto then $\Ker(f_1\mid f_2)=M_2$
\item[(iii)] $\Ker( \mathbf{0} \mid \mathbf{I})=0$
\item[(iv)] $\Ker(f_1\mid \mathbf{0})=M_{2}$
\item[(v)] In general, $\Img(f_2)\subseteq\Img(f_1)$ $\Rightarrow$ $\Ker(f_1\mid f_2)=M_2$
\item[(vi)] If $\Psi_2$ is an automorphism of $M_2$ then mapping $m_2\mapsto\Psi_2(m_2)$ is an isomorphism $\Ker(f_1\mid f_2\Psi_2)\rightarrow \Ker(f_1\mid f_2)$
\item[(vii)] $\Ker(f_1\mid f_2\Psi_2)=\Psi_2^{-1}(\Ker(f_1\mid f_2))$
\item[(viii)] If $\Psi_1$ is an automorphism of $M_1$ then $\Ker(f_1\Psi_1\mid f_2)=\Ker(f_1\mid f_2)$
\item[(ix)] If $\Psi$ is an isomorphism of $N$ then $\Ker(\Psi f_1\mid f_2)=\Ker(f_1\mid \Psi^{-1}f_2)$
\end{itemize}
\end{proposition}

\begin{proof}
$(i), (ii), (iii)$ are straightforward from Proposition \ref{f2-1Imf1}; $(iv), (v)$ are directly obtained from Definition \ref{DefKerKer}. 

Property $(vi)$ is consequence of the fact that $\Psi_2$ maps $\Ker(f_1\mid f_2\Psi_2)$ onto $\Ker(f_1\mid f_2)$. Since $\Psi_2$ is injective (because it is a restriction an isomorphism) it follows that $\Psi_2$ is itself an isomorphism and hence the result. Property $(vii)$ follows straightforward from $(vi)$.

Assertion $(viii)$ is clear because $\Ker(f_1\Psi_1\mid f_2)=\{m_2\in M_2 \mid \exists m_1\in M_1 : f_1(\Psi_1m_1)+f_2m_2=0\}$

Finally to prove $(ix)$ note that $\Ker(\Psi f_1\mid f_2)=\{m_2\in M_2 \mid \exists m_1\in M_1 : \Psi f_1(m_1)+f_2m_2=0\}$. Defining condition is equivalent to saying that $f_1m_1+\Psi^{-1}f_2m_2=0$ which is the definition of $\Ker(f_1\mid \Psi^{-1}f_2)$
\end{proof}

\section{Hereditary rings}

A $R$-module $M$ is called hereditary if every submodule of $M$ is projective. If every ideal of $R$ is projective (that is, $R$ is a ring of global dimension one) then every projective $R$-module is hereditary and commutative ring $R$ is called hereditary. A hereditary integral domain is a Dedekind domain.

Assume that $R$-modules $M_i$, $N$ are hereditary $R$-modules. It follows that all $R$-modules in exact sequence of Theorem \ref{TeoremaConucleo}
\begin{align*}
0\rightarrow \Ker(f_1)&\hookrightarrow \Ker(f_1, f_2)\rightarrow  \Ker(f_1 \mid f_2)\rightarrow 0\\
\end{align*}
are projective. Therefore
\begin{theorem}\label{FactorHereditario}
If $R$-modules $M_i$ and $N$ are hereditary (in particular if $R$ is itself an hereditary ring) then one has the factorization $\Ker(f_1,f_2)\cong\Ker(f_1)\oplus\Ker(f_1\mid f_2)$ and, consequently $$\Ker(f_1 \mid f_2)\cong \Ker(f_1, f_2)/\Ker(f_1)$$
\end{theorem}

\begin{example}
Put a matrix $A\in\mathbb{R}^{n\times n}$ with no real eigenvalues and $B\in\mathbb{R}^{n\times m}$ then one has the isomorphism $\Ker(z\mathbf{I}-A\mid B)\cong \Ker(z\mathbf{I}-A,B)$. This result, translated to the case of $z=\sigma$ the shift operator in $\mathbb{R}$-algebra $\mathcal{A}=\mathbb{R}^{\mathbb{Z}}$, implies that sequences of admissible controls $\vec{u}(t)$ of linear system $\vec{x}(t+1)=A\vec{x}(t)+B\vec{u}(t)$ are in bijective correspondece with trajectories (solutions) $(\vec{x}(t),\vec{u}(t))$ of linear system.
\end{example}


\section{Scalar Extension. Base change}

Assume in the sequel that $S$ is a $R$-algebra with structural morphism $\rho:R\rightarrow S$. Scalar extension of $R$-module $M$ is the $S$-module $\rho^{\ast}(M)=M\otimes_R S$, while scalar extension of $R$-linear map $f:M\rightarrow N$ is given by $\rho^{\ast}(f)=f\otimes\mathbf{I}_S:M\otimes_R S\rightarrow N\otimes_R S$.

Scalar extension functor $(-\otimes_R S)$ preserves onto linear mappings, it is right-exact (see \cite{Atiyah}) and hece it preserves epimorphisms. But notice that neither monomorphisms nor kernels are conserved in general by scalar extensions; for instance take $\Ker[2:\mathbb{Z}\rightarrow\mathbb{Z}]=0$ and observe that $\Ker[2\otimes\mathbf{I}_{\mathbb{Z}/2\mathbb{Z}}:\mathbb{Z}\otimes_{\mathbb{Z}}\mathbb{Z}/2\mathbb{Z}\rightarrow \mathbb{Z}\otimes_{\mathbb{Z}}\mathbb{Z}/2\mathbb{Z}]=\Ker[0:\mathbb{Z}/2\mathbb{Z}\rightarrow\mathbb{Z}/2\mathbb{Z}]=\mathbb{Z}/2\mathbb{Z}$

A $R$-module (respectively $R$-algebra) $S$ is flat (respectively flat algebra) if functor $(-\otimes_R M)$ happens to be exact; that is it transforms exact sequences into exact sequences. Examples of $R$-flat modules are free $R$-modules and projective $R$-modules \cite{Atiyah},\cite{Bass}, \cite{BOUR} or \cite{W}.

Within these conditions one has


\begin{theorem}\label{flat}
If $S$ is a flat $R$-algebra then
$$\rho^{\ast}(\Ker(f_1\mid f_2))=\Ker(\rho^{\ast}(f_1) \mid \rho^{\ast}(f_2))$$
\end{theorem}
\begin{proof}
Scalar extension functor $\rho^{\ast}(-)=-\otimes_R S$ is exact because $S$ is $R$-flat. Hence exact sequence 
$$
\xymatrix{
0\ar[r] &  \Ker(f_1) \ar[r]^{\iota}  & \Ker(f_1, f_2) \ar[r]^{\pi}& \Ker(f_1 \mid f_2) \ar[r] & 0 
}
$$ yields the exact sequence 
$$
\xymatrix{
0\ar[r] &  \Ker(f_1)\otimes_R S \ar[r]^{\iota\otimes \mathbf{I}_S}  & \Ker(f_1, f_2)\otimes_R S \ar[r]^{\pi\otimes \mathbf{I}_S} & \Ker(f_1 \mid f_2)\otimes_R S \ar[r] & 0 
}
$$

Once again, since $S$ is $R$-flat it follows the commutative square
$$
\xymatrix{
0\ar[r] &  \Ker(f_1)\otimes_R S \ar[r]^{\iota\otimes \mathbf{I}_S}\ar@{=}[d]  & \Ker(f_1, f_2)\otimes_R S\ar@{=}[d]\\
0\ar[r] &  \Ker(f_1\otimes \mathbf{I}_S) \ar[r]^{\iota'}  & \Ker(f_1\otimes \mathbf{I}_S, f_2\otimes \mathbf{I}_S)
}
$$
and therefore $\rho^{\ast}(\Ker(f_{1}| f_{2}))=\Ker(f_{1}| f_{2})\otimes_{R}S= \Ker(f_{1}\otimes \mathbf{I}_S\mid f_{2}\otimes \mathbf{I}_S)=\Ker(\rho^{\ast}(f_1) \mid \rho^{\ast}(f_2))$.

\end{proof}

\section{Product rings}
Now we deal with the case of $R=R_1\times\cdots\times R_t$ being a finite direct product of rings. In this case each factor ring $R_i\cong e_iR$ where $e_i=(0,...,0,1,0,...,0)$ is the $i$th structrual idempotent of product. Hence $R^1\cong e_1R\oplus\cdots\oplus e_tR$ and therefore each factor $R_i\cong e_iR$ is a flat $R$-algebra because it is projective due to it is a direct summand of the free $R$-module $R^1$. This decomposition can be traslated to kernels of pairs of linear maps:

\begin{theorem}\label{ProductoCuerpos}
If $\pi_i:R_1\times\cdots\times R_t\rightarrow R_i$ is the projection onto $i$th factor, then $\pi_i$ is also structural $R$-algebra morphism and, for a given pair of $R$-linear maps $f_j:M_j\rightarrow N$, $j=1,2$ one has
$$
\Ker(f_1\mid f_2)\cong\bigoplus_{i=1}^{t}\Ker(\pi_i^{\ast}(f_1)\mid\pi_i^{\ast}(f_2))
$$ 
\end{theorem}
\begin{proof}
The universal property of product yields the natural isomorphism $\Phi$ which is unique commutating both triangles in below natural diagram 

$$
\begin{xy}
(150,20)*+{\Ker(\pi_i^{\ast}(f_1)\mid\pi_i^{\ast}(f_2))}="Ker_i";
(150,0)*+{\Ker(\pi_j^{\ast}(f_1)\mid\pi_j^{\ast}(f_2))}="Ker_j";
(90,10)*+{\bigoplus_{i=1}^{t}\Ker(\pi_i^{\ast}(f_1)\mid\pi_i^{\ast}(f_2))}="OplusKer";
(30,10)*+{\Ker(f_1\mid f_2)}="Ker";
{\ar@{->}^{p_1} "OplusKer";"Ker_i"};
{\ar@{->}^{p_2} "OplusKer";"Ker_j"};
{\ar@{->}@/^{2pc}/^{\pi_i^{\ast}} "Ker"; "Ker_i"}
{\ar@{->}@/_{2pc}/_{\phi_j^{\ast}} "Ker"; "Ker_j"}
{\ar@{-->}^{\Phi}  "Ker";"OplusKer"}
\end{xy}
$$

\end{proof}

Now we conclude with the case of $R$ being a finite product of fields. This case contains modular rings $\mathbb{Z}/m\mathbb{Z}$ where $m$ is a square-free integer and would be of interest in convolutional coding.

\begin{corollary}\label{LocalGlobalKer}
Let $\mathbb{K}_i$ be a field for each $i$ and consider a ring $R=\mathbb{K}_1\times\cdots\times\mathbb{K}_s$. Let $A\in R^{p\times q_1}$ and $B\in R^{p\times q_2}$ be matrices. Then 
$$
\Ker(A\mid B)=e_1\cdot\Ker(\pi_1(A)\mid \pi_1(B))+\cdots+e_s\cdot\Ker(\pi_s(A)\mid \pi_s(B))
$$
\end{corollary}

Notice that Theorem \ref{TeoremaConucleo} applies to above case of $R$ beign a finite product of fields, hence one has the following result.

\begin{corollary}
Let $\mathbb{K}_i$ be a field for each $i$ and consider a ring $R=\mathbb{K}_1\times\cdots\times\mathbb{K}_s$. Let $A\in R^{p\times q_1}$ and $B\in R^{p\times q_2}$ be matrices. Then 
$$
\Ker(A\mid B)=\frac{\Ker(\pi_1(A), \pi_1(B))}{\Ker(\pi_1(A))}\oplus\cdots\oplus\frac{\Ker(\pi_s(A), \pi_s(B))}{\Ker(\pi_s(A))}
$$

\end{corollary}

Above results might be explained with an example. Consider the modular integer ring $\mathbb{Z}/30\mathbb{Z}\cong\mathbb{Z}/2\times\mathbb{Z}/3\mathbb{Z}\times\mathbb{Z}/5\mathbb{Z}$, where isomorphism is given by Chinese Remainder Theorem $$(a,b,c)\mapsto \text{unique }x\text{ such that }x=a(\text{mod }2),x=b(\text{mod }3),x=c(\text{mod }5)$$
Thence structural idempotents are given by:
$$
\xymatrix{
\mathbb{Z}/2\times\mathbb{Z}/3\mathbb{Z}\times\mathbb{Z}/5\mathbb{Z} \ar[r]  & \mathbb{Z}/30\\
(1,0,0) \ar[r] &  15\\
(0,1,0) \ar[r] &  10\\
 (0,0,1)\ar[r] &  6 
}
$$
Thus the kernel might be recovered from local data on every factor field by using Corollary \ref{LocalGlobalKer} with above structural idempotents $e_1=15$, $e_2=10$, and $e_3=6$.

Note that this result can also be generalized to the case of polynomial matrices which would be useful in the behavioral theory of linear systems and in particular to convolutional codes.

\begin{corollary}
Consider matrices $A(z), B(z)$ of adequate sizes, $p\times q_1$ and $p\times q_2$ respectively, and entries in $R[z]$ where $R=R_1\times\cdots\times R_s$ is a product ring with structural idempotents $e_i\in R$. Then trajectories can be computed locally and glued together; that is to say, one has
$$
\Ker(A(z)\mid B(z))=\{u(z)\mid\exists x(z):A(z)x(z)+B(z)u(z)=0\}=
$$
$$
=e_1\Ker(\pi_1(A)(z)\mid\pi_1(B)(z))+\cdots+e_s\Ker(\pi_s(A)(z)\mid\pi_s(B)(z))
$$
\end{corollary}
\begin{proof}
Note that $R[z]=(R_1\times\cdots\times R_s)[z]\cong R_1[z]\times\cdots\times R_s[z]$ and thus $R_i[z]$ is a flat $R[z]$-algebra. It only remains to detect structural morphism, which is a trivial exercise. Then we conclude by using Theorem \ref{ProductoCuerpos}
\end{proof}

To conclude note that if $\mathbb{K}$ is a field then $\mathbb{K}[z]$ is a principal ideal domain and thus a Dedekind domain and hence hereditary. Thus by applying Theorem \ref{FactorHereditario} one has the following factorization result.

\begin{corollary}
Consider matrices $A(z), B(z)$ of adequate sizes, $p\times q_1$ and $p\times q_2$ respectively, and entries in $R[z]$ where $R=\mathbb{K}_1\times\cdots\times \mathbb{K}_s$ is a product of fields. Then 
$$
\Ker((A(z)\mid B(z))=\frac{\Ker(\pi_1(A)(z)\mid\pi_1(B)(z))}{\Ker(\pi_1(A)(z))}\oplus\cdots\oplus\frac{\Ker(\pi_1(A)(z)\mid\pi_1(B)(z))}{\Ker(\pi_1(A)(z))}
$$
\end{corollary}

Notice that the key point in above results is some kind of "commutativity" between the product of rings and extension from a ring to its ring of polynomials; that is to say,$R[z]=(R_1\times\cdots\times R_s)[z]\cong R_1[z]\times\cdots\times R_s[z]$, which assures that $R_i[z]$ is a flat $R[z]$-algebra. This phenomenon is also true for the extensions to formal power series $R\mapsto R[[z]]$; polynomial ring $R[z,z^{-1}]$; formal power series $R[[z,z^{-1}]]$; Puisseux series; and, if $R_i$ are domains, for extensions to rational fractions $R\mapsto R(z)$ and Laurent series $R\mapsto R((z))$. Thus we conjecture that Theorem \ref{ProductoCuerpos} could be extended to these scenarios.

To conclude it is worth to note that factorization results \emph{\`a la} Corollary \ref{LocalGlobalKer} holds on by assuring that factor rings $R_i[z]$ (or respectively $R_i[[z]]$, ...) are hereditary.

\end{document}